\documentclass[12pt]{article}

\usepackage{amsmath, amsthm, amssymb, enumerate}

\theoremstyle{plain}
\newtheorem{theorem}{Theorem}[section]
\newtheorem{proposition}[theorem]{Proposition}

\theoremstyle{definition}
\newtheorem{definition}[theorem]{Definition}

\newcommand{\N}{{\mathbb{N}}}

\newcommand{\orb}{{\mathcal{O}}}

\begin{document}

\title{Devaney's chaos and eventual sensitivity}

\author{Alica MILLER\\ 
Department of Mathematics\\ 
University of Louisville\\
Louisville, KY 40292, USA\\ \medskip
{Email: \tt alica.miller@louisville.edu}
}
\date{}
\maketitle
\begin{abstract}  
We give an equivalent definition of Devaney chaotic semiflow in terms of eventual sensitivity, the notion recently introduced by C.~Good, R.~Leek, and J.~Mitchell. As a consequence, we prove a version of Auslander-Yorke dichotomy for the case of not necessarily compact phase spaces. Finally, we raise some questions about the relation between Devaney chaoticity and the notion of topological sensitivity, which was recently introduced by C.~Good and C.~Mac\'ias.


\end{abstract}

\section{Introduction}
The notion of a Devaney chaotic cascade $(X,f)$ on a metric space $X$ was introduced by R.~Devaney in 1989 in \cite{d}: a cascade $(X,f)$ is said to be Devaney chaotic if it is topologically transitive, has a dense set of periodic points, and is sensitive on initial conditions. In 1992 J.~Banks, J.~Brooks, G.~Cairns, G.~Davis, and P.~Stacey proved in \cite{bbcds} that every topologically transitive cascade on an infinite metric space $X$, which has a dense set of periodic points, is necessarily sensitive on initial conditions. This was quite surprising since the sensitivity was considered to be the main ingredient of Devaney's chaos. 

In this paper we show that a semiflow $(T,X)$ on a metric space $X$, with an arbitrary acting commutative topological monoid $T$ satisfies topological transitivity, density of periodic points, and sensitivity if and only if it is non-minimal, topologically transitive, and has a dense set of periodic points. The proof of one of the two directions is an adaptation of the proof of \cite[Theorem]{bbcds}. In this way we get two equivalent definitions of Devaney's chaoticity of a semiflow $(T,X)$. Our main theorem contains, besides these two definitions, a third one as well. Namely, in the 2019 paper \cite{glm} C.~Good, R.~Leek, and J.~Mitchell introduced the notion of an eventually sensitive cascade on a metric or compact Hausdorff space $X$. We use this notion in the context of a semiflow $(T,X)$ on a metric space $X$ and show in our main theorem that the previous two definitions of Devaney's chaoticity of $(T,X)$ are equivalent with $(T,X)$ being topologically transitive, having a dense set of periodic points, and being eventually sensitive. As a consequence we prove a version of Auslander-Yorke dichotomy, which is kind of analogous to  Generalized Auslander-Yorke Dichotomy II from \cite{glm},  but for the case when the phase space $X$ is not necessarily compact.

\section{Notation and preliminaries}
In this paper $T$ will always denote a {\sl noncompact} abelian topological monoid, written additively, whose identity element is $0$.
A subset $A$ of $T$ is called {\it syndetic} if there is a compact subset $K$ of $T$ (called a {\it corresponding compact} for $A$) such that $(t+K)\cap A\ne\emptyset$ for every $t\in T$.

\smallskip
All topological spaces in the paper are assumed to be Hausdorff. If $x$ is a point of a topological space $X$, we denote by $\mathcal{N}(x)$ the set of neighborhoods of $x$. If $X$ is a metric space, the metric will always be denoted by $d$. If $x$ is a point in a metric space $X$ and $r>0$, the {\it open ball} with center $x$ and radius $r$ is denoted by $B(x,r)$.

\smallskip 
A jointly continuous monoid  action $\phi:T\times X\to X$ of a topological monoid $T$ on a topological space $X$ is called a {\it semiflow} and denoted by $(T,X,\phi)$ or by $(T,X)$. The element $\phi(t,x)$ will be denoted by $t.x$ or $tx$, so that the defining conditions for a monoid action have the form
\begin{align*}
s.(t.x) &=(s+t).x,\\
0.x &=x,
\end{align*}
for any $s,t\in T$ and $x\in X$. The space $X$ is called the {\it phase space}. For any $x\in X$ the set $Tx=\{tx\;|\;t\in T\}$ is called the {\it orbit} of $x$ and is also denoted by $\orb(x)$. The set $\overline{Tx}$ is called the {\it orbit-closure of $x$}. The semiflow $(T,X)$ is called {\it minimal} (MIN) if $\overline{Tx}=X$ for every $x\in X$, and {\it non-minimal} (NMIN) otherwise. For $Y\subseteq X$ and $t\in T$ we denote $tY=\{ty\;:\;y\in Y\}$. If $f:X\to X$ is a continuous map, the semiflow $(\N_0, X)$, defined by $n.x=f^n(x)$ for all $n\in\N_0$ and $x\in X$, is called a {\it cascade}. Here $\N_0=\{0,1,2,\dots\}$ is equipped with the discrete topology.

\smallskip
A semiflow $(T,X)$ is called {\it topologically transitive} (TT) if for any nonempty open subsets $U,V$ of $X$ there is a $t\in T$ such that $tU\cap V\neq \emptyset$. 

\smallskip
A point $x$ in a semiflow $(T,X)$ is said to be {\it periodic} if its fixer $\mathrm{Fix}(x)=\{t\in T:tx=x\}$ is a syndetic sub\-monoid of $T$.

\medskip
A semiflow $(T,X)$ on a metric space $X$ is said to be {\it sensitive on initial conditions}, or {\it sensitive} (S),  if there is a number $c>0$ (called a {\it sensitivity constant}) such that
\[(\forall \,x\in X)\;(\forall\, \varepsilon>0)\;(\exists\, y\in B(x,\varepsilon))\; (\exists\, t\in T)\; d(tx,ty)\ge c.\]

\smallskip
The next definition was introduced in \cite{glm}. A semiflow $(T,X)$ on a metric space $X$ is said to be {\it eventually sensitive on initial conditions}, or {\it eventually sensitive} (ES), if there is a number $c>0$ (called an {\it eventual sensitivity constant}) such that 
\[(\forall\, x\in X)\;(\forall \,\varepsilon>0)\; (\exists\, t_0, t\in T)\; (\exists \,y\in B(t_0x,\varepsilon))\; d(t(t_0x),ty)\ge c.\]
If a semiflow $(T,X)$ is sensitive, it is obviously eventually sensitive. The converse is not true. The importance of eventual sensitivity is discussed in \cite{glm}, where, as well, an example of an eventually sensitive cascade which is not sensitive was provided.

\smallskip
Let $(T,X)$ be a semiflow on a metric space $X$ and $A\subseteq T$. We say that $A$ {\it acts equicontinuously at a point $x_0\in X$} if
\[(\forall\,\varepsilon>0)\;(\exists\, \delta>0)\; (\forall \,x\in B(x_0,\delta))\; (\forall\, t\in T)\; d(tx_0, tx)<\varepsilon\]
and that {\it $A$ acts uniformly equicontinuously on $X$} if
\[(\forall\,\varepsilon>0)\;(\exists\, \delta>0)\; (\forall\, x_1, x_2\in X)\; (\forall\, t\in T)\;  d(x_1, x_2)<\delta \Rightarrow d(tx_1, tx_2)<\varepsilon.\]
We say that the semiflow $(T,X)$ is {\it equicontinuous} (EQ) if $T$ acts equiconti\-nuously at every point of $X$
and that $(T,X)$ is {\it unformly equicontinuous} (UEQ) if $T$ acts uniformly equiconti\-nuously on $X$.

\smallskip
Part (a) of the next proposition is from \cite{km}, part (b) follows easily from (a) using the compactness of $X$.

\begin{proposition}\label{eq_prop}
Let $(T,X)$ be a semiflow on a metric space $X$. Let $A$ be a relatively compact subset of $T$. Then:

(a) $A$ acts equiconti\-nuously at every point of $X$.

(b) If $X$ is compact, $A$ acts unformly equicontinuously on $X$.
\end{proposition}

\medskip
The paper is self-contained, i.e., all the notions used in the paper are defined in it. The reader can also consult the standard reference \cite{w} and papers \cite{km, mm} for additional information. 

\section{Results}
The next proposition is from our paper \cite{mm}, written jointly with C.~Money. We include the proof for the sake of completeness.

\begin{proposition}\label{periodic_orbits}
Let $x\in X$ be a periodic point in a semiflow $(T,X)$. Then:

(a) For every $x'\in\orb(x)$, $\orb(x')=\orb(x)$.

(b) For every periodic point $x'\in X$, either $\orb(x)=\orb(x')$ or $\orb(x)\cap \orb(x')=\emptyset$.

(c) For every $x'\in \orb(x)$, $\mathrm{Fix}(x')=\mathrm{Fix}(x)$. In particular, all points from $\orb(x)$ are periodic.

(d) If $K$ is a compact subset of $T$ corresponding to the syndetic set $\mathrm{Fix}(x)$, then $\orb(x)=K.x$. In particular, $\orb(x)$ is a compact subset of $X$.

(e) Suppose $X$ is a metric space and $X=\orb(x)$. Then $(T,X)$ is a compact, minimal, uniformly equicontinuous semiflow.
\end{proposition}
\begin{proof}
Let $S=\text{Fix}(x)$ and let $K$ be a compact subset of $T$ whose every translate $t+K$,\, $t\in T$, intersects $S$.

(a) Let $x'\in\orb(x)$. Then $x'=t'.x$ for some $t'\in T$. Then there is a $k'\in K$ such that $t'+k'\in \mathrm{Fix}(x)$. Hence $k'.x'=k'.(t'.x)=(t'+k').x=x$, so that $x\in\orb(x')$. Hence $\orb(x)\subseteq \orb(x')$. Since also $x'\in\orb(x)$, we have $\orb(x')\subseteq \orb(x)$. Thus $\orb(x')=\orb(x)$.

(b) Follows from (a).

(c) Let $x'\in\orb(x)$, i.e., $x'=t'.x$ for some $t'\in T$. Then for any $s\in \mathrm{Fix}(x)$,\, $s.x'=s.(t'.x)=t'.(s.x)=t'.x=x'$, hence $\mathrm{Fix}(x)\subseteq \text{Fix}(x')$. By (a), $\orb(x')\subseteq \orb(x)$, in particular $x\in \orb(x')$. Hence, by symmetry, $\mathrm{Fix}(x')\subseteq\mathrm{Fix}(x)$. Thus $\mathrm{Fix}(x')=\mathrm{Fix}(x)$. In particular, $x'$ is periodic (having a syndetic fixer).

(d) Let $x'\in\orb(x)$. By (a), $x\in\orb(x')$, i.e., $t.x'=x$ for some $t\in T$. There is a $k\in K$ such that $t+k\in \mathrm{Fix}(x)$. Since, by (c), $\mathrm{Fix}(x')=\mathrm{Fix}(x)$, we have $t+k\in \mathrm{Fix}(x')$ as well. Hence $x'=(t+k).x'=k.(t.x')=k.x$. Thus $\orb(x)=K.x$ and, in particular, $\orb(x)$ is a compact set. 

(e) By (a) and (d), $(T,X)$ is a compact minimal semiflow. $K+K$ is a compact subset of $T$ and, since $X$ is also compact, it acts uniformly equicontinuously on $X$ by Proposition \ref{eq_prop}(b). Fix an $\varepsilon>0$. Let $\delta=\delta(\varepsilon, K+K)>0$ be such that 
\begin{equation}\label{eqn1}
(\forall\,x_1, x_2\in X)\,(\forall\,r\in K+K)\;d(x_1,x_2)<\delta \Rightarrow d(r.x_1, r.x_2)<\varepsilon. \\
\end{equation}
\noindent Fix $x_1, x_2\in X$ such that $d(x_1, x_2)<\delta$. Fix a $t \in T$. By (c), $\mathrm{Fix}(x_1)=\mathrm{Fix}(x_2)=\mathrm{Fix}(x)$. Let $k\in K$ and $s\in \mathrm{Fix}(x_1)$ be such that $t+k=s$. Since, by (b), $X=K.(k.x_1)$, there is a $k'\in\ K$ such that $t.x_1=k'.(k.x_1)=(k'+k).x_1$.
Hence $(2k+k').x_1=(t+k).x_1=s.x_1=x_1$ and so $2k+k'\in \mathrm{Fix}(x_1)=\mathrm{Fix}(x_2)$. Hence, by (c), $(2k+k').(t.x_2)=t.x_2$, whence $(k'+k).x_2=t.x_2$. Now by (\ref{eqn1}), $d(t.x_1, t.x_2)=d((k'+k).x_1, (k'+k).x_2)<\varepsilon$. 
Hence $(T,X)$ is uniformly equicontinuous.
\end{proof}

Next is the main theorem of the paper. The part $(a) \Rightarrow$ (b) generalizes the main theorem from \cite{bbcds} and the proof of it is an adaptation of the proof from  \cite{bbcds}.
The novelty is the part (c).

\begin{theorem}\label{main_thm}
Let $(T,X)$ be a semiflow on a metric space $X$. The following are equivalent:

(a) $(T,X)$ is (TT), (DPP), (NMIN);

(b) $(T,X)$ is (TT), (DPP), (S);

(c) $(T,X)$ is (TT), (DPP), (ES).
\end{theorem}

\begin{proof}
\underbar{(a) $\Rightarrow$ (b):} Suppose (a) holds.

{\bf Claim.} There exists a number $c>0$ such that for every $x\in X$ there is a periodic point $q\in X$ with $d(x, \orb(q))\ge 4c$.

\underbar{Proof of Claim.} Let $q_1, q_2$ be two periodic points with disjoint orbits. (They exist by (DPP), (NMIN), and Proposition \ref{periodic_orbits}.) Let $\displaystyle{d(\orb(q_1), \orb(q_2))=8c}$. Then by the triangle inequality for every $x\in X$ we have either $d(x, \orb(q_1))\ge 4c$, or $d(x, \orb(q_2))\ge 4c$. Claim is proved.

Let $c$ be as in Claim. We will show that $(T,X)$ is  sensitive with a sensitivity constant $c$. Let $x\in X$. Let $U$ be any neighborhood of $x$ and let $U_0=U\cap B(x,c)$. By (DPP) there is a periodic point $p\in U_0$. By Claim there is a periodic point $q$ such that $d(x, \orb(q))\ge 4c$. The fixer of $p$, $\mathrm{Fix}(p)$, is a syndetic subset of $T$. Let $K$ be a corresponding compact. By Proposition \ref{eq_prop}(a), $K$ acts equicontinuously at $q$. Hence there is a neighborhood $V$ of $q$ such that 
\begin{equation}\label{eqn2}
(\forall\, v\in V)\; (\forall\, k\in K)\; d(k.v, k.q)<c.
\end{equation}
By (TT), there is a point $y\in U_0$ and $t\in T$ such that $ty\in V$. For this $t$ let $k\in K$ be such that $t+k\in\mathrm{Fix}(p)$. Now,
\begin{align*}
d(x, (t+k).p) &= d(x,p)<c,\\
d((t+k).y, \orb(q)) &\le d(k.(ty), k.q))<c,
\end{align*}
the second inequality by (\ref{eqn2}). Hence, since $d(x, \orb(q))\ge 4c$, by the triangle inequality 
$d((t+k).p, (t+k).y)\ge 2c$. Hence, again by the triangle inequality,
either $d((t+k).x, (t+k).p)\ge c$, or $d((t+k).x, (t+k).y)\ge c$. Thus $(T,X)$ is sensitive.

\underbar{(b) $\Rightarrow$ (c):} Clear.

\underbar{(c) $\Rightarrow$ (a):} Suppose to the contrary, i.e., that $(T,X)$ is (TT), (DPP), (ES), and (MIN). Let $c>0$ be an eventual sensitivity constant. Let $p$ be a periodic point of $(T,X)$. Let $K$ be a compact corresponding to the syndetic subset $\mathrm{Fix}(p)$ of $T$. Then by Proposition \ref{periodic_orbits}, $\orb(p)=K.p$, a compact set, so that $X=\overline{Tp}=\orb(p)$. Hence, again by Proposition \ref{periodic_orbits}, $(T,X)$ is uniformly equicontinuous. Then for $c$ we have
\begin{equation}\label{eqn3}
(\exists\, \delta>0)\;(\forall\, x_1, x_2\in X)\; (\forall \,t\in T)\; d(x_1, x_2)<\delta \Rightarrow d(tx_1, tx_2)<c.
\end{equation}
Fix an $x\in X$. For these $\delta$ and $x$ let $t_0\in T$ be such that
\begin{equation}\label{eqn4}
(\exists\,y\in B(t_0x, \delta))\;(\exists\, t\in T)\;d(t(t_0x), ty)\ge c.
\end{equation}
But now (\ref{eqn4}) contradicts (\ref{eqn3}), so $(T,X)$ cannot be minimal.

The theorem is proved.
\end{proof}

The next theorem is a version of Auslander-Yorke dichotomy, which is in a way analogous to  Generalized Auslander-Yorke Dichotomy II from \cite{glm},  but for the case when the phase space $X$ is not necessarily compact.

\begin{theorem}
Let $(T,X)$ be a semiflow on a metric space $X$, which is (TT) and (DPP). Then it satisfies one and only one of the following two conditions:

(1) $(T,X)$ is (UEQ);

(2) $(T,X)$ is (ES).
\end{theorem}

\begin{proof}
Let $p$ be a periodic point of $(T,X)$. There are two cases possible, excluding each other: $X=\orb(p)$ and $X\ne\orb(p)$.

\underbar{First case: $X=\orb(p)$.} Then by Proposition \ref{periodic_orbits}, $(T,X)$ is (MIN), hence, by Theorem \ref{main_thm}, it is not (ES). Also, again by Proposition \ref{periodic_orbits}, it is (UEQ). 

\underbar{Second case: $X\ne\orb(p)$.} Then $(T,X)$ is (NMIN), hence, by Theorem \ref{main_thm}, it is (ES). Let $c$ be an eventual sensitivity constant. Suppose $(T,X)$ is (UEQ). Let $\delta>0$ be such that 
\begin{equation}\label{eq5}
(\forall\, x_1, x_2\in X)\; (\forall\, t\in T)\;  d(x_1, x_2)<\delta \Rightarrow d(tx_1, tx_2)<c.
\end{equation}
Fix an $x\in X$. By (ES), there are $t_0, t\in T$ and $y\in B(t_0x, \delta)$ such that $d(t(t_0x), ty)\ge c$. This contradicts to (\ref{eq5}). Thus $(T,X)$ is not (UEQ). 

The theorem is proved.
\end{proof}

\section{Conclusion}

Theorem \ref{main_thm} can be easily stated and proved in the more general context of uniform spaces. One also notices that Theorem \ref{main_thm} makes possible to define Devaney's chaos for the case of topological phase spaces $X$.

\begin{definition} \label{DC_top}
We say that a semiflow $(T,X)$ (on a topological space $X$) is {\it Devaney chaotic} if it is (TT), (DPP), and (NMIN).
\end{definition}

Since the other two definitions of Devaney's chaos (from Theorem \ref{main_thm}) can only be given in the context of uniform spaces (and metric ones, in particular), one can ask the following {\bf question}: {\it Suppose that a semiflow $(T,X)$ on a topological space $X$ is (TT), (DPP), and (NMIN). Is the topological space $X$ necessarily uniformizable?}

If it is, then in any uniform structure $\mathcal{U}$ on $X$, $(T,X)$ would be sensitive (by the uniform space version of Theorem \ref{main_thm}) and that would show that Devaney's chaos is a uniform and metric space notion and that sensitivity (defined for uniform spaces) is its essential ingredient. If it is not, then one can consider the following alternative definition of sensitivity that works for topological spaces, introduced in 2018 in \cite{gm} by C.~Good and C.~Mac\'ias.

\begin{definition}\label{GM-sensitivity}
We say that a semiflow $(T,X)$ on a topological space $X$ is {\it Good-Mac\'ias sensitive on initial conditions}, or {\it Good-Mac\'ias sensitive} (GMS), if there exists a finite open cover $\mathcal{U}$ of $X$ such that
\[(\forall\, x\in X)\;(\forall \,U\in\mathcal{N}(x))\;(\exists\, y\in U)\;(\exists\,t\in T)\;(\forall \,V\in\mathcal{U})\; |\{tx, ty\}\cap V|\le 1.\] 
\end{definition}
So we can now consider the following {\bf question}: {\it Suppose that a semiflow $(T,X)$ on a topological space $X$ is (TT), (DPP), and (NMIN). Is $(T,X)$ necessarily Good-Mac\'ias sensitive?}

Answering the previous two questions would give a better understan\-ding of the relation between Devaney's chaoticity and sensitivity of a general semiflow $(T,X)$. One may also consider the following {\bf question}: {\it Is it possible to define some kind of (Good-Mac\'ias) topological eventual sensitivity and what would be the relation between Devaney's chaoticity and that topological eventual sensitivity (in the spirit of Theorem \ref{main_thm})?}

\end{document}